\documentclass[10pt]{article}
\usepackage{amsmath}
\usepackage{amssymb}
\usepackage{indentfirst}
\usepackage[latin1]{inputenc}
\usepackage{geometry}
\usepackage{amsfonts}

\newcommand{\Z}{{\mathbb Z}}

\newcommand{\Q}{{\mathbb Q}}
\newcommand{\R}{{\mathbb R}}
\newcommand{\C}{{\mathbb C}}

\newcommand{\bc}{\begin{center}}
\newcommand{\ec}{\end{center}}

\newcommand{\U}{{\cal U}}

\newcommand{\qed}{\enspace\vrule  height6pt  width4pt  depth2pt}
\newcommand{\h}{{\mathcal{H}}}   
\newcommand{\oo}{{\mathfrak{o}}} 

\newcommand{\para}{\par\vspace{.25cm}}

\newenvironment{proof}{\par\noindent{\bf Proof.}}{$\qed$\par\bigskip}
\newtheorem{theorem}{Theorem}[section]
\newtheorem{definition}[theorem]{Definition}
\newtheorem{lemma}[theorem]{Lemma}
\newtheorem{corollary}[theorem]{Corollary}
\newtheorem{proposition}[theorem]{Proposition}
\newtheorem{remark}[theorem]{Remark}

\begin{document}

\title{ Free Groups in Quaternion Algebras
\thanks{Mathematics Subject Classification 
Primary [$16U60$, $20E05$]; Secondary [$16S34$, $20M05$].
\newline
Keywords: Hyperbolic Groups, Quaternion Algebras, Free Groups, Free Semigroups, Group Rings, Units , M\"obius transformation.
\newline Research supported by FAPESP(Funda\c c\~ao  de Amparo \`a 
Pesquisa do Estado de S\~ao Paulo), Proc. 2008/57930-1 and CNPq-Brazil} }
\author{Juriaans, S.O. \and Souza Filho, A.C.}
\date{}
\maketitle

\begin{abstract}
In \cite{jpsf} we constructed pairs of units  $u,v$ in $\Z$-orders of a
 quaternion algebra over $\Q (\sqrt{-d})$, $d \equiv 7 \pmod 8$ positive  and square free, such that $\langle u^ n,v^n\rangle$ is free for some $n\in \mathbb{N}$.  Here we extend this result to any imaginary  quadratic extension of $\ \mathbb{Q}$, thus including matrix algebras. More precisely, we show that $\langle u^n,v^n\rangle $ is a free group  for all $n\geq 1$ and $d>2$ and for $d=2$ and all $n\geq 2$. The units we use  arise from Pell's and Gauss' equations. A criterion  for a pair of homeomorphisms  to generate a free semigroup is also established and used to prove that two certain units generate a free semigroup but that, in this case, the Ping-Pong Lemma can not be applied to show that the group they generate is free.
\end{abstract}

\maketitle

\section{Introduction}

The constructions of free groups in algebras is a nontrivial problem and pursued by many researchers. Let $R$ be a ring with unity, $G$ a group, $RG$ the group ring over $R$,  $\U(RG)$ the group of units of $\ RG$ and $\U_1(RG)$ the subgroups of units of augmentation 1.  In \cite{hig}, Higman shows that, in case $G$ is finite and abelian, $\U_1(\Z G)= G \times F$, with  $F$ a finitely generated free abelian group. On the other hand, in \cite{hp}, it is proved that if $G$ is a  finite  non-Hamiltonian $2$-group, then $\U(\Z G)$ contains a free group. A result of Tits, \cite[Lemma 1.5.3]{uni}  gives necessary conditions for a pair of diagonalizable elements of $GL(2, \C)$ to generate a free group. This problem is also addressed in several other papers (see  \cite{gms}, \cite{smg} and \cite{ms}).

The problems of proving the existence of free groups and actually constructing them are two different tasks. In case of the integral group ring of a finite group, $\mathbb{Z}G$ say, the latter problem is related to the construction of a set of generators of the unit group of $\mathbb{Z}G$. Much work has been done in this direction and for most groups, generators are constructed for a subgroup of finite index in the unit group (see \cite{jl}). Actually, the typical case that is yet barely scratched is the case when $\mathbb{Q}G$ has a Wedderburn component which is a division ring, $\mathbb{D}$ say. One important case is when $\mathbb{D}$ is a quaternion algebra. Gon\c calves, Mandel and Shirvani (\cite{gms} ) were among the first to construct free groups in such algebras. But this was still far from constructing a set of generators of  the unit group of an order of a quaternion algebra. This was first achieved by Corrales, Jespers, Leal and del R\'\i o. They actually give a finite algorithm to compute a finite set of
generators of the unit group of an order in a non-split classical quaternion algebra $\mathbb{H}(\mathbb{K})$
 over an imaginary quadratic extension $\mathbb{K}$ of the rationals and then apply the algorithm to an explicit example in case $\mathbb{K}=\mathbb{Q}(\sqrt{-7})$. In doing so, they give the simplest example of a group ring for which no finite set of generators of a subgroup of finite index in its unit group was not yet known. This actual example is $\ \mathbb{Z} [\frac{1+\sqrt{-7}}{2}]K_8$, where $K_8$ is the quaternion group of order eight.

As showed in \cite{jpsf}, the problem above is related to the problem proposed by I. B.S. Passi in \cite{jpp}: {\it Let $G$ be a finite group and $R$ a  ring of characteristic zero. Describe $G$ and $R$ such that the unit group $\U (RG)$ is hyperbolic.} This problem was addressed in \cite{jpp} and later in \cite{jpsf}. As a consequence of the work done in \cite{jpsf},  two new constructions of units,  arising from solutions of Pell's and Gauss' equations,  were given   in the algebras $(\frac{-1,-1}{R})$, where $R=\oo_{\Q(\sqrt{-d})}$ is the ring of algebraic integers of a quadratic extensions of $\ \mathbb{Q}$. These units were  coined Pell and Gauss units.  Gauss units can have norm $-1$.  Hence,  one such Gauss unit and the set of generators of norm $1$ of \cite{cjlr}, give a full set of generators of the unit group of $\ \mathbb{Z} [\frac{1+\sqrt{-7}}{2}]K_8$. A result due to Gromov  shows that this unit group is a hyperbolic group with one end and its hyperbolic boundary is the  $2$-dimensional euclidean sphere. From this and \cite[Proposition III.$\Gamma.3.20$]{mbah}, it follows that if $d \equiv 7 \pmod 8$ is a positive integer, and $u,v\in \U((\frac{-1,-1}{R}))$ are Pell units with distinct supports  then $\langle u^n, v^n \rangle$ is a free group for a suitable positive $n$. 

The problem of determining the exponent $n$ is a non-trivial task and, in general, only its existence is guaranteed  (see for example \cite[Proposition $III.\Gamma.3.20$]{mbah}). Here we determine precisely all the exponents. The units,  used here, have suitable algebraic properties  which are used to  determine the exponent $n$, thus giving a generalization of \cite[Theorem $5.5$]{jpsf}. This is done for quaternion algebras over arbitrary  imaginary quadratic extensions of $\mathbb{Q}$, thus also including matrix algebras.  We note that neither Proposition $3.20$ of \cite{mbah} nor Tits' result can be applied because the norm of the eigenvalues of the units, when considered as matrices,  equals one.    Finally we  give a criterion for two homeomorphisms to generate a free semigroup and apply it to show  that two specific new units generate a free semigroup, while known criteria do not give that they generate a free group.


\section{Preliminaries}

We denote by  $(\frac{-1,\,-1}{K}):=K[i,\,j:i^{2}=-1,\,j^{2}=-1,\,-ji=ij=:k]$  the quaternion
algebra over $K$. Let $K$ be an algebraic  number field and $\mathfrak o_K$ its
ring of integers,  denote by $(\frac{-1,-1}{ o_K})$, the  $\mathfrak o_K$-span of    $\{1,\,i,\,j,k\}$, which is an $\mathfrak
o_K$-algebra. Clearly, since $o_K$ is a $\Z$-order of $K$, the algebra $(\frac{-1,-1}{ o_K})$ is a $\Z$-order of $(\frac{-1,\,-1}{K})$.

Denote by $\sigma$  the isomorphism $\sigma: \h \longrightarrow \h$, $\sigma(x)=j^{-1}xj$ with  $\h:=(\frac{-1,-1}{\C})$.
As a vector space,  $\h$ is isomorphic to $\C \oplus \C j$.   
Since,  elements $z \in \h$ have the form $z=x+yj$, where $x,y \in \C+\C[j]$, the $\mathbb{R}$-monomorphism 
$\Psi: \h \hookrightarrow M_{2}(\C)$, $\Psi(x+yj)= \left (
                                             \begin{array}{ll}x&y\\-\sigma(y)&\sigma(x)
                                             \end{array}\right )$ is well defined. Note that, when restricted to $\C$, $\sigma$ is the complex conjugation and thus  a quaternion algebra is isomorphic to the  complex matrices  of the form $\left (
                                             \begin{array}{ll}x&y\\-\overline{y}&\overline{x}, 
                                             \end{array}\right )$ (see \cite{bea}).

If $M=(m_{ij}) \in GL_2(\C)$, then $M$ defines a M\"obius transformation  $\varphi_M(z):=\frac{m_{11}z+m_{12}}{m_{21}z+m_{22}}$. Hence, $\Psi$ induces a homomorphism $\varphi: \h  \hookrightarrow  {\mathcal{ M}}$, $\varphi (u)=\varphi_u:=\varphi_{\Psi(u)}$ where ${\mathcal{ M}}$ is the group of M\"obius transformations. 

\typeout{We restrict the maps $\sigma$ and $\Psi$ on the quaternion algebra $(\frac{-1,-1}{\Q\sqrt{-d}})\subset \h$. Hence $\Psi((\frac{-1,-1}{\Q\sqrt{-d}})) \subset M_{2}(K)$, where $K=\Q(\sqrt{-1},\sqrt{-d})$ is the $\Q$-span of $\{1,\sqrt{-1},\sqrt{-d},\sqrt{d}\}$. }If $u=u_1+u_ii+u_jj+u_kk \in (\frac{-1,-1}{\Q(\sqrt{-d})})$ then  we have that   $\varphi ((u_1+u_i i)+(u_j +u_k i)j)=\left (
                                             \begin{array}{ll}u_1+u_i \sqrt{-1}&u_j+u_k \sqrt{-1}\\-u_j+u_k \sqrt{-1}&u_1-u_i \sqrt{-1}
                                             \end{array}\right ).$  
                                             
\para
 
In \cite{rob} it is shown that $\langle z+2, \frac{z}{2z+1}\rangle$  is a free subgroup of the group ${\mathcal{ M}}$. This is an application of the Ping-Pong Lemma, \cite[Lemma $3.19$]{mbah}, which states that if $h_1,\ \cdots, \ h_r$ are bijections of a set $\Omega$ and there exist non-empty disjoint subsets $A_{1,1},\ A_{1,-1},\ \cdots,\ A_{r,1},\ A_{r,-1} \subset \Omega$ such that $h_i^\epsilon(\Omega \setminus A_{i,\epsilon}) \subset A_{i,-\epsilon}$, for $\epsilon \in \{-1,1\}$ and $i=1,\ \cdots,\ r$, then $\langle h_1,\ \cdots,\ h_r\rangle $ generates a free subgroup of rank $r$ in Perm($\Omega$). In fact,  $h_1(z):=\frac{z}{2z+1}$ and $h_2(z):=z+2$  are bijections of $\Omega:=\R \cup \{\infty \}$. It is easily verified that the sets $A_{1,1}:=]-1,0[, \ A_{1,-1}:=[0,1]$, $\ A_{2,1}:=[-\infty,-1], \ \textrm{and} \ A_{2,-1}:=]1,\infty]$ are in the conditions of the Ping-Pong Lemma.

As a consequence we have that:

\begin{theorem} The group generated by the units $u=\sqrt{-1}+(-\sqrt{-1}+i)j$ and $w=\sqrt{-1}+(\sqrt{-1}+i)j$ in the algebra $(\frac{-1,-1}{\mathfrak o_{\Q(\sqrt{-1})}})\ $ is free.
\end{theorem}



\section{Free Groups in Quaternion Algebras}

In the sequel,  $K=\Q(\sqrt{-d})$  is  an imaginary quadratic
extension with  $d$ a positive and square-free integer. Let $\xi \neq \psi$ be elements of $\{1,\,i,\,j,\,k\}$.
 Suppose \begin{equation}u:=m\sqrt{-d} \xi+p \psi, \textrm{where}\  p,\,m\in \mathbb Z,\end{equation}
 is an element in $(\frac{-1,-1}{\mathfrak o_K})$ having  norm $1$.
 Then \begin{equation}\label{Pell}p^2-m^2d=1,\end{equation} i.e., the pair $(p,\,m)$ is a solution of
  Pell's equation $x^2-dy^2=1$ in $\Q(\sqrt{d})$.
  Equation (\ref{Pell}) implies that $\epsilon = p+m\sqrt{d}$
 is a unit in $\mathfrak o_{\Q(\sqrt{d})}$. Conversely, if $\epsilon=x+y\sqrt{d}$ is a unit of norm $1$ in $\mathfrak o_{\Q(\sqrt{d})}$
  then, necessarily, $x^2-y^2d=1$, and, therefore,  for any choice of $\xi,\,\psi$ in $\{1,\,i,\,j,\,k\}$, $\xi\neq \psi$, \begin{equation}\label{2pell}
 y\sqrt{-d}\xi+x\psi\end{equation} is a unit in $(\frac{-1,-1}{\mathfrak o_K})$. In particular, if \begin{equation}u_{(\epsilon,\,\psi)}:=x+y\sqrt{-d}\psi,\ \psi\in \{i,\,j,\,k\}, \end{equation} then $u_{(\epsilon,\,\psi)}$ is a unit in $(\frac{-1,-1}{\mathfrak o_K})$.
 \para
 If $u=u_1+u_ii+u_jj+u_kk$, we define $supp(u) = \{u_\xi \neq 0, \ \xi \in \{1,i,j,k\}\}$, the support of $u$.
 \para
 
 With the notations as above, we have:
 \para
\begin{proposition}\cite{jpsf}\label{gcpell}
\begin{enumerate}
\item If $1 \notin supp(u)$, the support of  $u$, then $u$ is a torsion unit.
\item If  $\epsilon=x + y\sqrt{d}$  is a unit in $\mathfrak o_{\Q(\sqrt{d})}$, then $$
 u^n_{(\epsilon,\,\psi)}=u_{(\epsilon^n,\,\psi)}$$for all $\psi\in \{i,\,j,\,k\} $ and  $n\in \mathbb Z$.
\end{enumerate}
\end{proposition}\para
\para

Units of type (\ref{2pell})  are called {\it Pell
$2$-units}. Denote the norm of $\epsilon=x+y\sqrt{d}$ by $\mathcal{N}( \epsilon):=\epsilon \overline{\epsilon}$, where $\overline{\epsilon}=x-y\sqrt{d}$. 
\para 

Likewise, we define the Pell $4$-unit 
$u:= \frac{y}{2}\sqrt{-d}\zeta+(\frac{y}{2}\sqrt{-d})\xi+(\frac{1 \pm x}{2})\psi+(\frac{1 \mp x}{2})\phi$, where $\zeta,\xi,\psi,\phi \in \{1,i,j,k\}$ are distinct two by two. Clearly $u \overline{u}=\frac{\mathcal{N}( \epsilon)+1}{2}$ is the norm of $u$, since $\mathcal{N}(\epsilon)=1$ we have that $u$ is a unit.
\para
 
 \begin{remark}$\empty$\\
\begin{itemize}
\item We cannot define a Pell $4$-unit when $\mathcal{N}(\epsilon)=-1$, since such units have norm $\frac{\mathcal{N}(\epsilon)+1}{2} \neq 0$.  
\item When $y \equiv 1 \pmod 2$, the invertible $\epsilon^2$ defines a Pell $4$-unit $u:=xy\sqrt{-d}+(xy\sqrt{-d})i+(x^2)j+(y^2d)k$ (see \cite{jpsf}).
\item If the invertible $\epsilon=(2x-1)+y\sqrt{2d} \in \Q(\sqrt{2d})$ has norm $\mathcal{N}( \epsilon)=1$, we define a Pell $3$-unit $u:=y\sqrt{-d} \xi + x \psi + (1-x) \phi$ in $(\frac{-1,-1}{\mathfrak o_{\Q( \sqrt{-d})}})$, where $\xi,\psi,\phi \in \{1,i,j,k\}$ are distinct.
\end{itemize}
\end{remark}

\begin{proposition} \label{pp1}
Let $\epsilon=x+y\sqrt{d}$ be the fundamental invertible in $\Q(\sqrt{d})$, $d>1$, with $\mathcal{N}(\epsilon)=1$. Then the  elements $x+(y\sqrt{-d})i$, $y\sqrt{-d}+(x)k$, $\frac{x+1}{2}-(\frac{y\sqrt{-d}}{2})i+(\frac{x-1}{2})j+(\frac{y\sqrt{-d}}{2})k$ and  $x^2-(xy\sqrt{-d})i-(y^2d)j+(xy\sqrt{-d})k\}$ are units in the quaternion algebra $(\frac{-1,-1}{\Q\sqrt{-d}})$
\end{proposition}
\begin{proof}
The first three elements are units. We only need to verify the element $w=x^2-(xy\sqrt{-d})i-(y^2d)j+(xy\sqrt{-d})k\}$ whose norm $u \overline{u}=(\mathcal{N}(\epsilon))^2=1$ and hence is a unit.
\end{proof}

We recall the well known three square Theorem of Gauss: {\it Let $n$ be a positive integer, where $n=4^a n^{'}$, $ 4  \nmid n^{'}$ and $0 \leq a$. Then  $n$ is a sum of three square integers if, and only if, $n^{'} {\not \equiv} 7 \pmod 8$ (see \cite{small}).} As a consequence of this result we have that if $d \equiv 7 \pmod 8$ is positive and $m \equiv 2 \pmod 4$ then there exist integers $p,q,r$, such that, $m \sqrt{-d} +p i+ q j+ r k$ is a unit of the quaternion algebra over $\Q \sqrt{-d}$.

\begin{definition}
Let $u:=m \sqrt{-d} +p i+ q j+ r k$ be a unit with $l:=|supp(u)|>1$, $m,p,q,r$ integers, $m \neq 0$, $d$ a positive square free integer and $(m^2d \pm 1)$ the sum of three square integers. Then  $u$ is called a Gauss unit or Gauss $l$-unit.
\end{definition}

Note that if $\epsilon=x+y\sqrt{d}$ has norm $\mathcal{N}(\epsilon)=-1$, then the unit $y\sqrt{-d}\xi+x\psi$, $\xi, \psi \in \{1, i,j,k\}$ and $\xi \neq \psi$, is a Gauss $2$-unit. We may  represent Gauss  and Pell units as M\"obius transformations. Using the  algebraic relation between the coefficients of the units, we will  construct free groups using suitable pairs of these units.

\begin{proposition}\label{mbs}
Let $u=x+(y\sqrt{-d})i$ and $w \in \{ \{y\sqrt{-d}+(x)k,\ \frac{x+1}{2}-(\frac{y\sqrt{-d}}{2})i+(\frac{x-1}{2})j+(\frac{y\sqrt{-d}}{2})k, \ x^2-(xy\sqrt{-d})i-(y^2d)j+(xy\sqrt{-d})k\}$ be the units of the Proposition \ref{pp1}. If we represent  $u$ and $w$ as a M\"obius transformation, then $\varphi_u(z)=\frac{x-y\sqrt{d}}{x+y\sqrt{d}}z$ is a homothety and $\varphi_w(z)=\frac{az+b}{cz+d}$, where $a,b,c,d$ are non-zero. 
\end{proposition} 
\begin{proof}
If $u=x+(y\sqrt{-d})i+0j$, then  $\Psi(u)=\left (
                                             \begin{array}{ll}x-y \sqrt{-d}\sqrt{-1}&0\\0&x+y \sqrt{-d}\sqrt{-1}
                                             \end{array}\right )$. Thus $\varphi_u(z)=\frac{x-y \sqrt{d}}{x+y \sqrt{d}}z$. If  $w=y\sqrt{-d}+((x)i)j$, hence $\Psi(w)=\left (
                                             \begin{array}{ll}y \sqrt{-d}&x \sqrt{-1}\\x \sqrt{-1} &y \sqrt{-d}
                                             \end{array}\right )$. Thus $\varphi_w(z)=\frac{y\sqrt{d}z+x}{xz+y\sqrt{d}}$. Proceeding this way, if $w=\frac{x+1}{2}-(\frac{y\sqrt{-d}}{2})i+(\frac{x-1}{2}+(\frac{y\sqrt{-d}}{2})i)j$, then $\varphi_w(z)=\frac{(y\sqrt{d}+(x+1))z-(y\sqrt{d}-(x-1))}{-(y\sqrt{d}+(x-1))z-(y\sqrt{d}-(x+1))}$. If $w=x^2-(xy\sqrt{-d})i+(-y^2d+(xy\sqrt{-d})i)j$, then $\varphi_w(z)=(\frac{x+y\sqrt{d}}{x-y\sqrt{d}})\frac{xz-y\sqrt{d}}{-(y\sqrt{d})z+x}$.
\end{proof}

 By  the previous  proposition, if $\epsilon=x+y\sqrt{d}$ is the fundamental invertible and $u=x+(y\sqrt{-d})i$, then
 $\varphi_u(z)=\frac{x-y\sqrt{d}}{x+y\sqrt{d}}z \in \mathcal{M}$. Since $x,y$ are always positive, if $\mathcal{N}(\epsilon)=1$, then $x+y\sqrt{d}>1$ and 
 $0<x-y\sqrt{d}<1$. Hence, we have that $\varphi_u(z)=\rho z$ with $\rho=\frac{x-y\sqrt{d}}{x+y\sqrt{d}} \in ]0,1[$.

  Let  $\varphi(z)=\frac{az+b}{cz+d}$, leaving  $\Omega :=\mathbb{R}\cup {\infty}$ invariant, and $z_p:=\frac{-d}{c}$ and $z_0:=\frac{-b}{a}$, respectively its  pole and  zero. Given a Pell unit $w$, the pole $z_p$ of $\varphi \in \{\varphi_w,\varphi_w^{-1}\}$ plays an important role in the arguments we will use in the sequel.    
  
\begin{lemma}\label{ntrvl}
Let $w$ be as in the previous proposition, $z_p$ the pole of $\varphi_w$ and $I=]a,b[ \subset \Omega$ an interval. If $z_p \in I$ then $J:=\varphi_w(\Omega \setminus I)$ is an interval. Furthermore, $J$ is either $[\varphi_w(a),\varphi_w(b)]$ or $[\varphi_w(b),\varphi_w(a)]$.
\end{lemma}
\begin{proof} This is a consequence of $\Omega=\R \cup \{\infty\}$
\end{proof}

%

\begin{proposition} \label{gnrl}
Let $\varphi_w$ be the M\"obius transformations of the Proposition \ref{mbs}, $z_p$(respectively $z^{'}_p$) the pole and $z_0$(respectively $z^{'}_0$) the zero of $\varphi_w$(respectively $\varphi^{-1}_w)$. If $\varphi_w(z)=\frac{y\sqrt{d}z+x}{xz+y\sqrt{d}}$ then $I:=]\varphi_w(\frac{3}{2}z_0),\varphi_w(\frac{1}{2}z_p)[ \subset [-\frac{1}{2}z_p,-\frac{3}{2}z_0]$ and  
$J:=]\varphi^{-1}_w(-\frac{1}{2}z_p),\varphi^{-1}_w(-\frac{3}{2}z_0)[  \subset [\frac{3}{2}z_0,\frac{1}{2}z_p]$. If $\varphi_w(z)=\frac{(y\sqrt{d}+(x+1))z-(y\sqrt{d}-(x-1))}{-(y\sqrt{d}+(x-1))z-(y\sqrt{d}-(x+1))}$, then $I:=]\varphi_w(3z_p),\varphi_w(\frac{1}{2}z_0)[ \subset [3z^{'}_p,\frac{1}{2}z^{'}_0]$ and  
$J:=]\varphi^{-1}_w(-\frac{1}{2}z^{'}_0),\varphi^{-1}_w(3z^{'}_p)[  \subset [\frac{1}{2}z_0,3z_p]$.
 If $\varphi_w(z)=(\frac{x+y\sqrt{d}}{x-y\sqrt{d}})\frac{xz-y\sqrt{d}}{-(y\sqrt{d})z+x}$, then $I:=]\varphi_w(3z_p),\varphi_w(\frac{1}{2}z_0)[ \subset [3z^{'}_p,\frac{1}{2}z^{'}_0]$ and  
$I:=]\varphi^{-1}_w(-\frac{1}{2}z^{'}_0),\varphi^{-1}_w(3z^{'}_p)[  \subset [\frac{1}{2}z_0,3z_p]$.
\end{proposition} 
\begin{proof}
If $\varphi_w(z)=\frac{y\sqrt{d}z+x}{xz+y\sqrt{d}}$ then $-\frac{x}{y\sqrt{d}}=z_0<z_p=-\frac{y\sqrt{d}}{x}$ and hence $I=]\frac{xy\sqrt{d}}{x^2+2},\frac{x^2+1}{xy\sqrt{d}}[$ and $J=]-\frac{(x^2+1)}{xy\sqrt{d}},-\frac{xy\sqrt{d}}{x^2+2}[$. We have to prove that $I \subset [\frac{y\sqrt{d}}{2x},\frac{3x}{2y\sqrt{d}}]$ which is equivalent to $J \subset [-\frac{3x}{2y\sqrt{d}},-\frac{y\sqrt{d}}{2x}]$. Clearly, 
$\frac{y\sqrt{d}}{2x}<\frac{xy\sqrt{d}}{x^2+2}<\frac{x^2+1}{xy\sqrt{d}}<\frac{3x}{2y\sqrt{d}}$ since $1<x$. 
If $\varphi_w(z)=\frac{(y\sqrt{d}+(x+1))z-(y\sqrt{d}-(x-1))}{-(y\sqrt{d}+(x-1))z-(y\sqrt{d}-(x+1))}$, then $\varphi^{-1}_w(z)=\frac{-(y\sqrt{d}-(x+1))z+(y\sqrt{d}-(x-1))}{(y\sqrt{d}+(x-1))z+(y\sqrt{d}+(x+1))}$ and $0<z_0=\frac{y\sqrt{d}-(x-1)}{y\sqrt{d}+(x+1)}<\frac{-(y\sqrt{d}-(x+1))}{y\sqrt{d}+(x-1)}=z_p$ and  $z^{'}_p=  \frac{-(y\sqrt{d}+(x+1))}{y\sqrt{d}+(x-1)} <\frac{y\sqrt{d}-(x-1)}{y\sqrt{d}-(x+1)}=z^{'}_0<0$. Let $()$ be an order relation, then 
$\frac{(y\sqrt{d}+(x+1))}{y\sqrt{d}+(x-1)} ()- \frac{y\sqrt{d}-(x-1)}{y\sqrt{d}-(x+1)}$, is the relation between positive numbers, we have $-(y^2d-(x+1)^2) () y^2d -(x-1)^2$ hence $2x^2 () 0$ and $()$ is the relation $>$, since the numbers are negative we change the relation to $<$ . We have the intervals $I=]\frac{-(x+2)}{y\sqrt{d}},\frac{-y \sqrt{d}}{x+3}[$ and $J=]\frac{-x^2+xy\sqrt{d}+1}{x+3},\frac{x+2}{x^2+xy\sqrt{d}-1}[$ and we have to prove that $I \subset [3\frac{-(y\sqrt{d}+(x+1))}{y\sqrt{d}+(x-1)},\frac{1}{2}\frac{y\sqrt{d}-(x-1)}{y\sqrt{d}-(x+1)}]$ and $J \subset [\frac{1}{2}\frac{y\sqrt{d}-(x-1)}{y\sqrt{d}+(x+1)},3\frac{-(y\sqrt{d}-(x+1))}{y\sqrt{d}+(x-1)}]$. With the relation between the positive numbers  $3\frac{(y\sqrt{d}+(x+1))}{y\sqrt{d}+(x-1)}()\frac{(x+2)}{y\sqrt{d}}$, we obtain $2x^2-x+2()-y\sqrt{d}(2x+1)$ and thus $()$ is the relation $>$. Proceeding this way, we show the inclusion of the interval $I$. For the interval $J$, we compare the left bounds $\frac{-x^2+xy\sqrt{d}+1}{x+3}()\frac{1}{2}\frac{y\sqrt{d}-(x-1)}{y\sqrt{d}+(x+1)}$ and obtain $y\sqrt{d}()\frac{(x-1)^2}{x-3}$, with the algebraic equation of $\epsilon$, if $x>1$ then  we have that () is the relation $<$. Similarly we verify that $\frac{x+2}{x^2+xy\sqrt{d}-1}<3\frac{-(y\sqrt{d}-(x+1))}{y\sqrt{d}+(x-1)}$. If $\varphi_w(z)=(\frac{x+y\sqrt{d}}{x-y\sqrt{d}})\frac{xz-y\sqrt{d}}{-(y\sqrt{d})z+x}$, then $\varphi^{-1}(z)=\frac{\rho xz+y\sqrt{d}}{\rho y \sqrt{d}z+x}$, where $\rho=\frac{x-y\sqrt{d}}{x+y\sqrt{d}}$ the routine is the same, $\frac{y\sqrt{d}}{x}=z_0<z_p=\frac{x}{y\sqrt{d}}$, $-\frac{1}{\rho}\frac{x}{y\sqrt{d}}=z^{'}_p<z^{'}_0=-\frac{1}{\rho}\frac{y\sqrt{d}}{x}$ and we obtain $I=]\frac{-(2x^2-1)}{2xy\sqrt{d}},\frac{-xy\sqrt{d}}{x^2+1}[$ and $J=]\frac{1}{\rho}\frac{xy\sqrt{d}}{x^2+1},\frac{1}{\rho}\frac{(2x^2-1)}{2xy\sqrt{d}}[$ and the inclusions are clear. \end{proof}

\begin{proposition} \label{hmth}
Let $\varphi_u(z)=\frac{x-y \sqrt{d}}{x+y \sqrt{d}}z$, $\varphi_w$ and  $z_p,z^{'}_p,z_0,z^{'}_0$ as in the last proposition. If $w=y\sqrt{-d}+(x)k$, then $I=[\varphi_u(\frac{3}{2}z_0), \varphi_u(-\frac{3}{2}z_0)] \subset ]\frac{1}{2}z_0),-\frac{1}{2}z_0)[$, otherwise $I=[\varphi_u(3z^{'}_p), \varphi_u(3z_p)] \subset ]\frac{1}{2}z^{'}_0),-\frac{1}{2}z_0)[$.
\end{proposition} 
\begin{proof}
If $w=y\sqrt{-d}+(x)k$ we obtain $z_0$ and $z_p$ as in the Proposition \ref{gnrl}. Hence $I=[-\rho \frac{3x}{2y\sqrt{d}},\rho \frac{3x}{2y\sqrt{d}}]$ and we have to prove that $I \subset [-\frac{y\sqrt{d}}{2x},\frac{y\sqrt{d}}{2x}]$. Let $()$ be an order relation, then $\rho \frac{3x}{2y\sqrt{d}}()\frac{y\sqrt{d}}{2x}$ and we have $\frac{3x^2}{y^2d}()\frac{1}{\rho}$.  Clearly, $\frac{3x^2}{y^2d}=\frac{3y^2d+3}{y^2d}=3+\frac{3}{y^2d}<5$ and $5<\frac{1}{\rho}=x^2+y^2d+2xy\sqrt{d}$. Thus $\rho \frac{3x}{2y\sqrt{d}}<\frac{y\sqrt{d}}{2x}$ and the inclusion is proved. If $\varphi_w(z)=\frac{(y\sqrt{d}+(x+1))z-(y\sqrt{d}-(x-1))}{-(y\sqrt{d}+(x-1))z-(y\sqrt{d}-(x+1))}$, the interval $I=[\varphi_u(3z^{'}_p), \varphi_u(3z_p)]$ is determined and we compare the numbers $3\frac{y\sqrt{d}+(x+1)}{y\sqrt{d}+(x-1)}()\frac{1}{2}\frac{y\sqrt{d}-(x-1)}{-y\sqrt{d}+(x+1)}$ which is reduced to $\frac{5x+7}{7x+5}()\frac{y\sqrt{d}}{x}$ and if $x>2$ then $\varphi_u(3z^{'}_p)<\frac{1}{2}z^{'}_0$. Likewise, we prove that $\varphi_u(3z_p)<-\frac{1}{2}z_0$ and the inclusion is proved. For the next unit $w$, the proof goes along th same lines.  \end{proof} 
 
 These results and the Ping-Pong Lemma will be used in the proof of our next result.

\begin{theorem}
Let $\epsilon=x+y\sqrt{d}$ be the fundamental invertible in $\Q(\sqrt{d})$ with $\mathcal{N}(\epsilon)=1$ and  $u=x+(y\sqrt{-d})i$. If  $w\in \{y\sqrt{-d}+(x)k,\ \frac{x+1}{2}-(\frac{y\sqrt{-d}}{2})i+(\frac{x-1}{2})j+(\frac{y\sqrt{-d}}{2})k, \ x^2-(xy\sqrt{-d})i-(y^2d)j+(xy\sqrt{-d})k\}$.
Then  $\langle u,w \rangle $ is a free subgroup of $\  \U((\frac{-1,-1}{\oo_K}))$.
\end{theorem}

\begin{proof} 
Consider  $\varphi_u, \varphi_w \in \mathcal{M}$  as homeomorphisms of $\Omega$.  We claim that there exist real numbers $a_2<a_1<0<b_1<b_2$ such that $$A_{1,1}:=[a_2,a_1], \
 A_{1,-1}:=[b_1,b_2], \ A_{2,1}:=[-\infty, a_2[ \cup ]b_2,\infty] \ \textrm{and} \ A_{2,-1}:=]a_1,b_1[$$ are sets satisfying the conditions stated in the Ping-Pong Lemma.

Let $h_2(z):=\varphi_u(z)=\frac{x-y\sqrt{d}}{x+y\sqrt{d}}z$ and $h_1(z):=\varphi_w(z)$, where $w$ is one of the units of the theorem. Since each $A_{i,\pm 1}, \ i=1,2$ is to be an interval, or a disjoint union of two intervals, we get the following  conditions.

\begin{description}
\item[first] $h_1(\Omega \setminus A_{1,1}) \subset A_{1,-1}$, which is equivalent to  $ h_1(a_2),h_1(a_1) \in A_{1,-1}$;
\item[second]   $h_1^{-1}(\Omega \setminus A_{1,-1}) \subset A_{1,1} $, which is equivalent to  $h_1^{-1}(b_1),h_1^{-1}(b_2) \in  A_{1,1}$;
\item[third]   $h_2(\Omega \setminus A_{2,1}) \subset A_{2,-1} $, which is equivalent to  $ h_2(a_2),h_2(b_2) \in  A_{2,-1};$
\item[fourth]  $h_2^{-1}(\Omega \setminus A_{2,-1}) \subset A_{2,1} $, which is equivalent to  $h_2^{-1}(a_1),h_2^{-1}(b_1) \in A_{2,1}.$
\end{description}

 \noindent Suppose first that  $w=y\sqrt{-d}+(x)k$ and set $h_1(z):=\frac{x}{y\sqrt{d}}\frac{\frac{y\sqrt{d}}{x}z+1}{\frac{x}{y\sqrt{d}}z+1}$.  Clearly,  $z_p=\frac{-y\sqrt{d}}{x}>z_0=\frac{1}{z_p}$. Set $a_2=-b_2:=\frac{3}{2}z_0$ and $a_1=-b_1:=\frac{1}{2}z_p$, hence $a_2<z_0<z_p<a_1<0<b_1<b_2$. By   Lemma \ref{ntrvl} the set $h_1(\Omega\setminus A_{1,1})$ is an interval.  The first condition:  $h_1(\Omega\setminus A_{1,1})=]h_1(a_2),h_1(a_1)[ \subset A_{1,-1}$.  
 By  Proposition \ref{gnrl} it holds that  $h_1(a_2)=\frac{xy\sqrt{d}}{x^2+2}$ and $h_1(a_1)=\frac{x^2+1}{xy\sqrt{d}}$. Hence, we have $b_1<h_1(a_2)$. Similarly,  $h_1(a_1)=\frac{x^2+1}{xy\sqrt{d}}<\frac{3x}{2y\sqrt{d}}=b_2$, and thus $h_1(\Omega\setminus A_{1,1})\subset A_{1,-1}$. The second condition and Lemma \ref{ntrvl}: $h^{-1}_{1}(\Omega\setminus A_{1,-1})=]h_1^{-1}(b_1),h_1^{-1}(b_2)[ \subset A_{1,1}$. In fact, by Proposition \ref{gnrl}, $h_1^{-1}(b_2)=\frac{-xy\sqrt{d}}{x^2+2}=-h_1(a_2)$ and $h_1^{-1}(b_1)=\frac{-x^2-1}{xy\sqrt{d}}$. Clearly,  $h_1^{-1}(b_2)<a_1$ and $a_2<h_1^{-1}(b_1)$, and hence  $h^{-1}_{1}(\Omega\setminus A_{1,-1})\subset A_{1,1}$. The third and fourth conditions: $h_2(\Omega \setminus A_{2,1})=[h_2(a_2),h_2(b_2)] \subset A_{2,-1}$.  Clearly $h_2(b_2)=-h_2(a_2)$ and by  Proposition \ref{hmth} we have $h_2(b_2)<b_1$ and $a_1<h_2(a_2)$. Since $h_2(z)=\rho z$ is a homothety with $\rho>0$, then clearly $b<h_2(a)$ if, and only if, $h_2^{-1}(b)<a$, and hence we proved that $h_2(\Omega \setminus A_{2,1}) \subset A_{2,-1} $ and $h_2^{-1}(\Omega \setminus A_{2,-1})=[-\infty,h_2^{-1}(a_1)]\cup[h_2^{-1}(b_1),\infty] \subset A_{2,1} $.

Second, suppose that $w=\frac{x+1}{2}-(\frac{y\sqrt{-d}}{2})i+(\frac{x-1}{2})j+(\frac{y\sqrt{-d}}{2})k$. In this case the pole and zero of  $h_1(z)=\frac{(y\sqrt{d}+(x+1))z-(y\sqrt{d}-(x-1))}{-(y\sqrt{d}+(x-1))z-(y\sqrt{d}-(x+1))}$ are, respectively,  $z_p=\frac{-(y\sqrt{d}-(x+1))}{y\sqrt{d}+(x-1)}$ and $z_0=\frac{y\sqrt{d}-(x-1)}{y\sqrt{d}+(x+1)}$, where $z_0<z_p$ and the pole and the zero  
of $\varphi^{-1}_w$ are, respectively,  $z^{'}_p=\frac{-(y\sqrt{d}+(x+1))}{y\sqrt{d}+(x-1)}$ and $z^{'}_0=\frac{y\sqrt{d}-(x-1)}{y\sqrt{d}-(x+1)}$, where $z^{'}_p<z^{'}_0$. Since $z_p$ is positive, we now define the intervals  by $A_{1,-1}=[a_2,a_1]$ and $A_{1,1}:=[b_1,b_2]$ with
 $a_2:=3z^{'}_p$, 
 $a_1:=\frac{z^{'}_0}{2}$ 
 and
  $b_1:=\frac{z_0}{2}$,  
  $b_2:=3z_p$, and proceed as before proving that $h_1(\Omega \setminus A_{1,1}) =]h_1(b_2),h_1(b_1)[\subset A_{1,-1}$ and  $h_1^{-1}(\Omega \setminus A_{1,-1}) =]h_1^{-1}(a_1),h_1^{-1}(a_2)[\subset A_{1,1} $ which is a consequence of   Proposition \ref{gnrl}, because the congruence $y \equiv 0 \pmod 2$ implies that  $x>2$. The third and forth conditions are consequence of  Proposition \ref{hmth} and the fact that $h_2$ is a homothety. 

Finally, suppose that $w=x^2-(xy\sqrt{-d})i-(y^2d)j+(xy\sqrt{-d})k$. Then we have that  $h_1(z)=(\frac{x+y\sqrt{d}}{x-y\sqrt{d}})\frac{xz-y\sqrt{d}}{-(y\sqrt{d})z+x}$. 
 Set $b_1:=\frac{z_0}{2}=\frac{y\sqrt{d}}{2x}$, 
 $b_2:=3z_p=\frac{x}{y\sqrt{d}}$, 
 $a_1:=\frac{z_0^{'}}{2}=-\frac{b_1}{\rho}$ 
 and $a_2:=3z^{'}_p=-\frac{b_2}{\rho}$, where $\rho=\frac{x-y\sqrt{d}}{x+y\sqrt{d}}$ and the intervals $A_{1,-1}:=[a_2,a_1]$ and $A_{1,1}:=[b_1,b_2]$.
 By Lemma \ref{ntrvl}, the sets $h_1(\Omega \setminus A_{1,1})$ and $h_1^{-1}(\Omega \setminus A_{1,-1})$ are intervals. By  Proposition \ref{gnrl}, have that $h_1(\Omega \setminus A_{1,1}) =]h_1(b_2),h_1(b_1)[\subset A_{1,-1}$ and  $h_1^{-1}(\Omega \setminus A_{1,-1}) =]h_1^{-1}(a_1),h_1^{-1}(a_2)[\subset A_{1,1} $. The third and forth conditions are a consequence of  Proposition \ref{hmth} and the fact that $h_2$ is a homothety.

  Since all conditions are satisfied we have, by the   Ping-Pong Lemma, that  $\langle u,w \rangle$  is a free group.  \end{proof}

A natural question that can be raised is whether   the previous theorem still holds if the norm of the fundamental invertible is $-1$.  The answer is positive for a Gauss $2$-unit with $x \neq 1$.  When $\epsilon=1+\sqrt{2}$, the same calculations as before can be used to show that $\langle u^2,w\rangle$ is  a free group. To see this, we apply the the Ping-Pong Lemma using the following data: $-a_2=b_2:=2z_p=2\sqrt{2}$, $-a_1=b_1:=\frac{z_0}{2}=\frac{1}{2\sqrt{2}}$,  $A_{1,1}:=[a_2,a_1], \
A_{1,-1}:=[b_1,b_2], \ A_{2,1}:=[-\infty, a_2[ \cup ]b_2,\infty] \ \textrm{and} \ A_{2,-1}:=]a_1,b_1[$ obtaining that  $\langle u^{2},w \rangle$  is free. 

\begin{corollary} If $\mathcal{N}(x+y\sqrt{d})=-1$  and $x \neq 1$ then $\langle u,w \rangle$  is free, where $w$ is the Gauss $2$-unit $y\sqrt{-d}+xk$.
\end{corollary}

\begin{proof} 
From $x^2-y^2d=-1$ we have that $y\sqrt{d}>x$.  We apply the proof of the previous theorem using the following data: 
 $-a_2=b_2:=\frac{3}{2}\frac{y\sqrt{d}}{x}$ and $-a_1=b_1:=\frac{1}{2}\frac{x}{y\sqrt{d}}$. The condition  $x \neq 1$ is equivalent to $\epsilon \neq 1+\sqrt{2}$.
\end{proof} 

The next proposition elucidates the fact that the method above does not give us that $\langle 1+(\sqrt{-2})i,\sqrt{-2}+k\rangle$ is a free group.

\begin{proposition} Let $h_1$ and $h_2$ be the M\"obius transformation induced by the units $w=\sqrt{-2}+k$ and $u=1+(\sqrt{-2})i$, respectively. For the partition $A_{1,1}:=[a_2,a_1], \
A_{1,-1}:=[b_1,b_2], \ A_{2,1}:=[-\infty, a_2[ \cup ]b_2,\infty] \ \textrm{and} \ A_{2,-1}:=]a_1,b_1[$, if $-a_2$ and $b_2$ are greater than $\sqrt{2}$ and  $-a_1$ and $b_1$ are less than $\frac{1}{\sqrt{2}}$, then $h_1$ and $h_2$ do not verify the conditions of the Ping-Pong Lemma.
\end{proposition}
\begin{proof}
Clearly, by Proposition \ref{mbs}, $h_1(z)=\frac{\sqrt{2}z+1}{z+\sqrt{2}}$ and $h_2(z)=\frac{1-\sqrt{2}}{1+\sqrt{2}}z$. Since $h_i(\Omega \setminus A_{i,1}) \subset A_{i,-1} \iff h_i^{-1}(\Omega \setminus A_{i,-1})\subset A_{i,1}$, for $i=1,2$ it is sufficient to verify the conditions $h_i(\Omega \setminus A_{i,1}) \subset A_{i,-1}$ for $i=1,2$. We obtain the restrictions 
$\left \{ \begin{array}{l} 
\frac{1-\sqrt{2}}{1+\sqrt{2}}a_2 <\frac{\sqrt{2}a_1+1}{a_1+\sqrt{2}} \\
\frac{\sqrt{2}a_2+1}{a_2+\sqrt{2}}< \frac{1+\sqrt{2}}{1-\sqrt{2}}a_1 \end{array} \right .$, which has no solution on the conditions that $a_2<-\sqrt{2}$ and $\frac{-1}{\sqrt{2}}<a_1$.

\end{proof}

We made use of the algebraic properties of the Pell units. This allowed us  precise control of images in ${\mathcal{M}}$  of the units involved  so that the Ping-Pong Lemma could be applied. 

Without appealing to the use of the Ping-Pong Lemma, we give a criterion for two homeomorphisms  $\varphi_1,\varphi_2$ to generate a free semigroup. Recall that if $\varphi:V \longrightarrow V$ and $U \subset V$, then $U$ is invariant under $\varphi$ if $\varphi (U)\subset U$.

\begin{lemma}\label{mb}
Let $V$ be a set of infinite cardinality  and let 
$\varphi_1,\varphi_2:V \longrightarrow V$ be injective maps of infinite order.
 If  $U \varsubsetneqq  V$ is invariant under $\varphi_1$ and $\varphi_2$,  and 
 $x_0 \in V \setminus U$ is a fixed point of  $\varphi_1$ such that $\varphi_2(x_0) \in U$, then $\langle \varphi_1,\varphi_2 \rangle$ is a free semigroup.
\end{lemma}

\begin{proof}  Suppose that the reduced word $\varphi=\varphi_{r_1}^{s_1}\cdots \varphi_{r_k}^{s_k}$ is the identity map. If $r_k=1$, then, since $\varphi$ is a reduced word and $U$ is invariant under  both maps, $x_0=\varphi(x_0)=\varphi_{r_1}^{s_1}\cdots \varphi_{r_k}^{s_k}(x_0)$ $=\varphi_{r_1}^{s_1}\cdots \varphi_{r_{k-1}}^{s_{k-1}}(x_0)\in U$.  If we have that $r_k=2$, then  $x_0=\varphi(x_0) \in U$, since $\varphi_2(x_0)\in U$.  In any case we have a contradiction because $x_0 \notin U$.  
\end{proof}

We now give an application of the above result in the context of  Pell and Gauss units.

\begin{theorem}
Let $\epsilon=x+y\sqrt{d}$ be the fundamental invertible of $\Q(\sqrt{d})$  and  $u=x+(y\sqrt{-d})i$. If  $w \in \{y\sqrt{-d}+xk,\ \frac{x+1}{2}-\frac{y\sqrt{-d}}{2}i+\frac{x-1}{2}j+\frac{y\sqrt{-d}}{2}k,\ x^2-(xy\sqrt{-d})i-y^2dj+(xy\sqrt{-d})k\}$. 
Then  $\langle u,w \rangle \subset \U((\frac{-1,-1}{\oo_K}))$ is a free semigroup.
\end{theorem}

\begin{proof} Consider  $\varphi_1, \varphi_2 \in \mathcal{M}$, $V=\mathbb{R}\cup \{\infty \}$, $U=\mathbb{R}^*_+$, the set of positive real numbers,  and  $x_0=0$. We claim that the previous lemma can be applied using these data.

\noindent Suppose first that $w=y\sqrt{-d}+xk$. The map $\varphi_1(z)=\varphi_w(z)=\frac{x-y\sqrt{d}}{x+y\sqrt{d}}z$ clearly keeps $U$ invariant and fixes $x_0=0\notin U$. $\varphi_2(z)=\frac{x}{y\sqrt{d}}\frac{\frac{y\sqrt{d}}{x}z+1}{\frac{x}{y\sqrt{d}}z+1}$ has a pole at  $-\frac{y\sqrt{d}}{x}$, a zero at  $\frac{-x}{y\sqrt{d}}$ and $\varphi ({\infty})=\frac{y\sqrt{d}}{x}$. If $\mathcal{N}(\epsilon)=1$, then  
 $\varphi_2(0)=\frac{x}{y\sqrt{d}}>\frac{y\sqrt{d}}{x}>0$.  Hence, $\varphi_2(0)\in U$ and $\varphi_2 (U\cup\{0\}) =\ ]\frac{y\sqrt{d}}{x},\frac{x}{y\sqrt{d}}]\subset U$. If $\mathcal{N}(\epsilon)=-1$, then $\frac{y\sqrt{d}}{x}>\frac{x}{y\sqrt{d}}=\varphi_2(0)$. Hence $\varphi_2(0) \in U$ and $\varphi_2(U \cup \{0\})=[\frac{x}{y\sqrt{d}},\frac{y\sqrt{d}}{x}[ \subset U$. 

Second, suppose $w=\frac{x+1}{2}-\frac{y\sqrt{-d}}{2}i+\frac{x-1}{2}j+\frac{y\sqrt{-d}}{2}k$.  Set  $\varphi_2(z):=\varphi_w^{-1}(z)=\frac{-(y\sqrt{d}-(x+1))z+(y\sqrt{d}-(x-1))}{(y\sqrt{d}+(x-1))z+(y\sqrt{d}+(x+1))}$.   If $z_p$ and $z_0$ are, respectively, the pole and the zero of $\varphi_2$, then  $z_0<z_p$. Also, $\varphi_2(\infty )= \frac{-(y\sqrt{d}-(x+1))}{(y\sqrt{d}+(x-1))} <\frac{y\sqrt{d}-(x-1)}{y\sqrt{d}+(x+1)}=\varphi_2(0)$. Hence,  $\varphi_2(U)=]\varphi_2(\infty ),\varphi_2(0)[ \subset U$. 

Finally, if $w=x^2-(xy\sqrt{-d})i-y^2dj+(xy\sqrt{-d})k$, define $\varphi_2(z):=\varphi_w^{-1}(z)=\frac{(x^2-xy\sqrt{d})z+(y^2d+xy\sqrt{d})}{(-y^2d+xy\sqrt{d})z+(x^2+xy\sqrt{d})}$. The proof follows as in the previous item.

  Since all the conditions of the previous lemma  are satisfied, it follows that $\langle \varphi_1,\varphi_2 \rangle$  is a free semigroup.
\end{proof}

Let $u$ and $w$ be units in  $(\frac{-1,-1}{\oo_{\Q (\sqrt{-2})}})$. It follows from this theorem that if $u=1+(\sqrt{-2})i$ and $w=\sqrt{-2}+k$ then $<u,w>$ is a free semigroup, while we cannot decide whether the group generated by these units is free or not.

\bibliographystyle{amsalpha}
\par

$$\begin{array}{lll}
\textrm{Instituto de Matem\'atica e Estatística} && \textrm{Escola de Artes, Ci\^encias e Humanidades}\\
\textrm{Universidade de S\~ao Paulo(IME-USP)}  & & \textrm{Universidade de S\~ao Paulo (EACH-USP)}\\
\textrm{Caixa Postal 66281 }       & & \textrm{Rua Arlindo B\'ettio, 1000, Ermelindo Matarazzo}\\
\textrm{S\~ao Paulo, CEP 05315-970 - Brasil}                &  & \textrm{S\~ao Paulo, CEP 03828-000 - Brasil}\\
\textrm{email: ostanley@usp.br}               &   & \textrm{email: acsouzafilho@usp.br}
\end{array}$$

\end{document}